\documentclass[12pt, oneside]{amsart}

\usepackage{hyperref}
\hypersetup{%
	pdfpagemode={UseOutlines},
	bookmarksopen,
	pdfstartview={FitH},
	colorlinks,
	linkcolor={blue},
	citecolor={blue},
	urlcolor={blue}
}

\usepackage{geometry} %for the margins

\usepackage{dcolumn}
\newcolumntype{d}{D{.}{.}{-1} } %to vetical align numbers in tables, along the decimal dot

\usepackage{amsmath, amssymb, amsthm}
\usepackage{mathrsfs}
\usepackage{float}
\usepackage{esint}
\usepackage{enumitem}
\usepackage{tikz}
\usepackage{hyperref}
\usepackage{systeme}
\usepackage{mathtools}
\hypersetup{colorlinks, linkcolor=blue, citecolor=blue}
\usepackage{cancel}
\usepackage{aligned-overset}
\usepackage{color}
\usepackage[numbers]{natbib}

\numberwithin{equation}{section}

\newtheorem{teo}{Theorem}[section]

\newtheorem{cor}[teo]{Corollary}
\newtheorem{prop}[teo]{Proposition}
\newtheorem{conj}[teo]{Conjecture}

\theoremstyle{remark}
\newtheorem{remark}[teo]{Remark}

\newcommand{\R}{\mathbb{R}} %scorciatoia per R reali
\newcommand{\V}{\mathcal{V}} %STFT
\newcommand{\W}{\mathcal{W}} %Wavelet transform
\newcommand{\C}{\mathbb{C}} %Complex numbers
\renewcommand{\H}{\mathcal{H}} %Hilbert space
\renewcommand{\L}{\mathscr{L}} %sesquilinear form localization operator

\DeclareMathOperator*{\esssup}{ess\,sup}
\title[The Hilbert-Schmidt norm of localization operators]{The quantitative isoperimetric inequality for the Hilbert-Schmidt norm of localization operators}
\author{FABIO NICOLA AND FEDERICO RICCARDI}

\begin{document}

	\maketitle
	\begin{center}
		\it Dedicated to Karlheinz Gr\"ochenig, on the occasion of his 65th birthday
	\end{center}
	\begin{abstract}
		\noindent In this paper we study the Hilbert-Schmidt norm of time-frequency localization operators $L_{\Omega} \colon L^2(\R^d) \rightarrow L^2(\R^d)$, with Gaussian window, associated with a subset $\Omega\subset\R^{2d}$ of finite measure. We prove, in particular, that the Hilbert-Schmidt norm of $L_\Omega$ is maximized, among all subsets $\Omega$ of a given finite measure, when $\Omega$ is a ball and that there are no other extremizers. Actually, the main result is a quantitative version of this estimate, with sharp exponent. A similar problem is addressed for wavelet localization operators, where rearrangements are understood in the hyperbolic setting.
	\end{abstract}
	
	\section{Introduction}
	The short-time Fourier transform (STFT) of a function $f \in L^2(\R^d)$ with respect to a window $\varphi \in L^2(\R^d)$ is defined (see e.g. Gr\"ochenig's book \cite{grochenig}) as
	\begin{equation}\label{eq STFT}
		\V_{\varphi} f(x,\omega) = \int_{\R^d} f(y) \overline{\varphi(y-x)} e^{- 2 \pi i \omega \cdot y} \, dy, \quad (x,\omega) \in \R^{d}\times\R^{d}.
	\end{equation}
	A common choice for the window is the $L^2$-normalized Gaussian, that is
	\begin{equation*}
		\varphi(x) = 2^{d/4} e^{- \pi |x|^2}, \quad x \in \R^d.
	\end{equation*}
	In this note we will always consider this window and therefore we will simply set $\V=\V_{\varphi}$.
	
	Since we chose $\varphi$ normalized in $L^2(\R^d)$, we have that $\V : L^2(\R^d) \rightarrow L^2(\R^{2d})$ becomes an isometry. Hence, if $\|f\|_{2}=1$, the quantity $|\V f(x,\omega)|^2$, known as spectrogram, can be interpreted as the \emph{time-frequency energy density} of $f$ around the point $(x,\omega)$ in the time-frequency space. With this in mind, it is clear why having good and meaningful estimates for the short-time Fourier transform (and in particular for the spectrogram) has always been of great importance both from a theoretical and practical point of view. One of the first and at the same time most important results in this sense was obtained by Lieb in 1978 \cite{lieb_entropy} and is known today as Lieb's uncertainty inequality, namely
	\begin{equation}\label{eq lieb_intro}
		\| \V f \|_p^p \leq \left(\frac{2}{p}\right)^d \|f\|_2^p
	\end{equation}
	for every $f \in L^2(\R^d)$ and $2\leq p <\infty$ (see also \cite{carlen} for the identification of the extremal functions, and \cite{lieb1990} for generalizations). Lieb's inequality is a global estimate. In the spirit of uncertainty principles, we may be interested also in local estimates, that is, for some $\Omega \subset \R^{2d}$ with finite Lebesgue measure, finding bounds of the quantity
	\begin{equation*}
		\frac{\int_{\Omega} |\V f (x,\omega)|^2 \, dxd\omega}{\|f\|_2^2},
	\end{equation*}
	which represents the fraction of energy of $f$ contained in $\Omega $. The above integral can be written in an equivalent way as follows
	\begin{equation*}
		\int_{\Omega} |\V f (x,\omega)|^2 \, dx\,d\omega = \langle \chi_{\Omega} \V f, \V f \rangle = \langle \V^* \chi_{\Omega} \V f, f \rangle, 
	\end{equation*}
	where the operator 
	\[
	L_\Omega:=\V^* \chi_{\Omega} \V
	\]
	 naturally appears. This interpretation reveals a connection between time-frequency energy concentration estimates and the properties of the operator $L_\Omega$. In particular, since $\Omega$ has finite measure, $L_\Omega$ is a compact self-adjoint nonnegative operator (see e.g. \cite{wong}) and therefore its operator norm is given by
	\begin{equation*}
		\|L_\Omega\|=\max_{f \in L^2(\R^d) \setminus \{0\}} \dfrac{\langle \V^* \chi_{\Omega} \V f, f \rangle}{\|f\|_2^2}.
	\end{equation*}
	Hence, maximizing the norm of $\V^* \chi_{\Omega} \V$ corresponds to maximize the energy fraction of any function $f \in L^2(\R^d)\setminus\{0\}$ on $\Omega$. In this connection, Tilli and the first author \cite{nicolatilli_fk} recently proved that, among all subsets $\Omega$ of a given finite measure, $\|L_\Omega\|$ is maximum when $\Omega$ is a ball and that there are no other extremizers (a more general conjecture of Abreu and Speckbacher \cite{abreu2021donoho} was also proved in \cite{nicolatilli_fk}). We address to \cite{frank2023sharp,kalaj2,kalaj1,kulikov,ortega,ramos-tilli} for extensions of this result to other geometries -notably the hyperbolic and spherical one- and  for applications in complex analysis. See also \cite{grochenigLCA,nicolaLCA,nicola2023uncertainty} for similar problems on locally compact Abelian groups and to \cite{abreu-doerfler,abreu2021donoho,dias,grochenig2003,trapasso,seip} for related work.
	 
	In general, one can also measure the time-frequency concentration by a weighted $L^2$-norm, hence considering, for a function $F \colon \R^{2d} \rightarrow \C$, the so-called \emph{time-frequency localization operator}
	\begin{equation*}
		L_{F} \coloneqq \V^* F \V
	\end{equation*}
(hence $L_{\chi_\Omega}=L_\Omega$, with a slight abuse of notation). Since their first appearance in \cite{berezin} and \cite{daubechies}, time-frequency localization operators were intensively studied; see, for example, \cite{abreu, cordero, daubechies, fernandez, luef, wong} and the references therein for general results concerning boundedness, compactness, Schatten properties and asymptotics of the eigenvalues. Also, Lieb's uncertainty inequality \eqref{eq lieb_intro} can be equivalently rephrased, by duality, as 
\[
\|L_F\|\leq (1/p')^{d/p'}\|F\|_p, \qquad 1<p<\infty,
\]
 $p'$ being the conjugate exponent. Similar estimates in case the weight $F$ is taken in the intersection of Lebesgue spaces, with a full characterization of the extremal functions, were recently considered in 
 \cite{galbis2022norm,nicolatilli_norm, riccardi2023}. 
 	
	In this paper we address similar problems for the Hilbert-Schmidt norm of time-frequency localization operators, especially of the kind $L_{\Omega}$. An initial result, which follows from Riesz' rearrangement inequality, states that
	\begin{equation}\label{eq : introduction qualitative estimate}
		\| L_{\Omega} \|_{\mathrm{HS}} \leq \| L_{\Omega^*} \|_{\mathrm{HS}},
	\end{equation}
	where $\Omega^*$ is the (open) ball centered at $0$, with the same measure as $\Omega$, and that equality occurs if and only if $\Omega$ is (equivalent, up to a set of measure zero, to) a ball (see Proposition \ref{prop : HS norm is increasing after rearrangement}). In Section \ref{sec : HS norm of wavelet localization operators} we also prove an analogous result for wavelet localization operators, and also for general localization operators $L_F$.
	
	However, our interest is towards a \emph{quantitative} version of the previous estimate. In general, quantitative estimates are stability results for geometric and functional inequalities stating that if a function is ``almost optimal'' for some inequality then it must be ``close'' to the set of the corresponding optimizers. This kind of results have been proved for lots of different inequalities, such as the isoperimetric inequality, Sobolev and Gagliardo-Nirenberg inequality. For a comprehensive survey on the topic, see \cite{figalli2013stability} and the references therein.
	
	Only recently, quantitative estimates have been addressed for certain time-frequency concentration problems. Precisely the quantitative version for the above mentioned Faber-Krahn type result \cite{nicolatilli_fk} was addressed in \cite{gomez2023stability}, whereas the quantitative version of Lieb's uncertainty inequality \eqref{eq lieb_intro} (and for the generalized Wehrl entropy of mixed states) was proved in  \cite{frank2023generalized}.
	
In this note we want to fit into this thread by focusing on the following question:
	\begin{center}
		{\it If a set $\Omega \subset \R^{2d}$ ``almost'' attains equality in \eqref{eq : introduction qualitative estimate}	\\ can we conclude that $\Omega$ is ``almost'' a ball?}
	\end{center}
	The answer is positive and is given in Proposition \ref{prop : quantitative estimate for HS norm}, where the following quantitative bound is stated:
	\begin{equation}\label{eq : intro quantitative estimate}
		c_1 \beta(|\Omega|) \alpha[\Omega]^2 \leq \| L_{\Omega^*} \|_{\mathrm{HS}}^2 - \|L_{\Omega}\|_{\mathrm{HS}}^2,
	\end{equation}
	where $\beta$ is given by
	\begin{equation*}
		\beta(t) = \left\{
		\begin{aligned}
			t^{2+\frac{1}{d}},\quad	     & \text{for } 0 < t \le 1\\
			t^2 e^{-c_2 t^{1/d}}, \quad  & \text{for } t > 1
		\end{aligned}\right.
	\end{equation*}
	with $c_1,c_2 > 0$, and $\alpha[\Omega]$ is the \emph{Fraenkel asymmetry index}, which is defined as
	\begin{equation}\label{eq fraenkel}
		\alpha [\Omega] \coloneqq \inf \left\{ \dfrac{| \Omega \triangle B|}{|\Omega|} \colon B \subset \R^{2d} \text{ is a ball of measure } |B| = |\Omega| \right\}.
	\end{equation}
Here $\Omega \triangle B = (\Omega \setminus B) \cup (B \setminus \Omega)$ is the symmetric difference of $\Omega$ and $B$. We observe that $\alpha[\Omega]$ is a dimensionless quantity, and that the above infimum is achieved by some (not necessarily unique) ball.  

In our proof we use a quantitative version of Riesz' rearrangement inequality proved by Christ in \cite{christ2017sharpened} and we follow the strategy used by Frank and Lieb in \cite{frank2019proof} to address an optimization problem for the potential energy functional in $\R^d$ with interaction kernel $|x|^{-\lambda}$, $0< \lambda < d$. This connection with physically relevant problems is actually not that surprising, since our issue can be seen as a similar problem for a potential energy with Gaussian interaction (cf. \eqref{eq : formula HS norm localization operators with characteristic function} below). We observe that this type of isoperimetric problems dates back at least to Poincaré \cite{poincare} and still represents a challenging and very active research field (see, for example, \cite{burchard2015geometric, burchard2018nonlocal, figalli2015isoperimetry, frank2019proof} and the references therein).
	
	In Section \ref{sec : some remarks} we analyze the optimality of \eqref{eq : intro quantitative estimate} and in particular we prove that the exponent 2 of $\alpha[\Omega]$ and the behavior of $\beta(t)$ for $t \to 0^+$ are sharp, while for $t \to +\infty$ we conjecture that the estimate actually could hold with $\beta(t) = t^{2-1/2d}$ (Conjecture \ref{conj 1}). This leads us to another conjecture (Conjecture \ref{conj 2}), that is a refinement of Christ's result and seems of independent interest.
	
We now know, in particular, that both the operator norm and the Hilbert-Schmidt norm of $L_\Omega$ are maximized, among all subsets $\Omega$ of a given finite measure, when $\Omega$ is a ball. It is natural to wonder whether the same holds for other Schatten-von Neumann norms. We plan to investigate this issue, together with the above mentioned conjectures, in a subsequent work. 
	
	\section{Notation and preliminaries}\label{sec : notation and preliminaries}
		In the following, we are going to denote the ball with center 0 and radius $r$ in $\R^d$ or $\R^{2d}$ (depending on the context) as $B_r$. The $d$-dimensional Lebesgue measure of a subset $\Omega \subset \R^d$ will be denoted as $|\Omega|$. We set $\|f\|_p$ for the $L^p$ norm of a function $f$. The Fourier transform of a function $f$ will be denoted by $\widehat{f}$.
		
		Given a subset $\Omega \subset \R^d$ of finite measure, we denote by $\Omega^*$ the \emph{symmetric rearrangement of the set} $\Omega$, that is the open ball with center 0 and such that $|\Omega|=|\Omega^*|$. In the spirit of the layer cake representation, given a nonnegative function $f$ on $\R^d$ we can also define the \emph{symmetric-decreasing rearrangement of} $f$ as
		\begin{equation*}
			f^*(x) = \int_0^{+\infty} \chi_{\{|f|>t\}^*}(x) \, dt.
		\end{equation*}
		The symmetric-decreasing rearrangement appears in various optimization problems and inequalities. In the following, we are going to use one of the main results in this sense, that is the \emph{Riesz' rearrangement inequality}.
		\begin{teo}[Riesz' rearrangement inequality]{(\cite[Theorem 3.7, Theorem 3.9]{liebloss})}\label{th : Riesz rearrangement inequality}
			Let $f,g$ and $h$ be three nonnegative measurable functions on $\R^d$. Then we have
			\begin{equation}\label{eq : Riesz rearrangement inequality}
				\int_{\R^d \times \R^d} f(x) g(x-y) h(y) \, dx dy \leq \int_{\R^d \times \R^d} f^*(x) g^*(x-y) h^*(y) \, dx dy,
			\end{equation}
			with the understanding that if the left-hand side is $+\infty$ then also the right-hand side is. If, in addition, $g$ is strictly symmetric decreasing and $f$ and $h$ are not zero and the above integrals are finite, equality occurs if and only $f(x) = f^*(x-y)$ and $h(x)=h^*(x-y)$ for almost every $x\in\mathbb{R}^d$ and some $y \in \R^d$.
		\end{teo}
		Rearrangements are possible not only in the Euclidean setting, but also for other geometries. In the following, we will consider the Poincaré upper half-plane $\R \times \R_+ = \{(x,s) \in \R^2 \colon s>0\} \simeq \C_+$, endowed with the hyperbolic distance
		\begin{equation*}
			d_H(z,w) = 2 \, \mathrm{arctanh} \left\lvert \dfrac{z-w}{z-\overline{w}} \right\rvert, \quad z,w \in \C_+
		\end{equation*}
		and the hyperbolic measure given by $d\nu = dxds/s^2$, that is the left Haar measure of $\R \times \R_+$ regarded as the affine (``$ax+b$'') group. Recalling that the unit of the group is $(0,1)$ we can easily define the symmetric rearrangement of a subset $E \subset \R \times \R_+$, like in the Euclidean case, as the ball $E^* = \{z \in \C_+ \colon \varrho(z, (0,1)) < r\}$, where $r$ is chosen so that $\nu(E) = \nu(E^*)$. We point out that hyperbolic balls with center $(0,1)$ and radius $R$ are, as subsets of $\R^2$, Euclidean balls with center $(0,\cosh(R))$ and radius $\sinh(R)$. Then, given a nonnegative measurable function $f$ on $\C_+$ we can define its symmetric decreasing rearrangement exactly as we have done in the Euclidean case.
		
		In Section \ref{sec : HS norm of wavelet localization operators} we are going to need the hyperbolic version of Theorem \ref{th : Riesz rearrangement inequality}, which holds with the proper adjustments also in the hyperbolic setting (see \cite[Section 7.6]{baernstein2019symmetrization}).
		\begin{teo}\label{th : Riesz rearrangement inequality hyperbolic}
			Let $f,h$ be two nonnegative measurable functions on $\C_+$ and let $g \colon [0,+\infty) \rightarrow [0,+\infty)$ be decreasing. Then we have
			\begin{equation*}
				\int_{\C_+ \times \C_+} f(z) g(d_H(z,w)) h(w) \, d\nu(z) d\nu(w) \leq \int_{\C_+ \times \C_+} f^*(z) g(d_H(z,w)) h^*(w) \, d\nu(z) d\nu(w),
			\end{equation*}
			with the understanding that if the left-hand side is $+\infty$ then also the right-hand side is. If, in addition, $g$ is strictly symmetric decreasing and $f$ and $h$ are not zero and the above integrals are finite, equality occurs if and only $f(z) = f^*(az+b)$ and $h(z)=h^*(az+b)$ for almost every $z \in \C_+$ and some $a>0$ and $b \in\R$.
		\end{teo}
		Observe that $az+b$ is just the product (in the affine group) of $(b,a)$ and $z$. 
	\section{Hilbert-Schmidt norm of localization operators}
	\subsection{A general setting for localization operators}
	In this section we define localization operators in a general setting. 
	
	Given a $\sigma$-finite measure space $(X, \mu)$ and a separable Hilbert space $\H$, with norm $\| \cdot \|$ induced by the inner product $\langle \cdot , \cdot \rangle$, we suppose to have a map $ X\ni x \mapsto \varphi_x \in \H$ such that:\medskip
	\begin{enumerate}[label*=(\roman{*})]
		\item\label{coherent state map is weakly measurable} the map $X\ni x \mapsto \langle f, \varphi_x \rangle$ is measurable for every $f \in \H$;\medskip
		\item\label{coherent state map is bounded} $\| \varphi_x \| \leq c_1$ for a.e. $x \in X$ and some constant $c_1 > 0$.
	\end{enumerate}
	\medskip
	Then, for $F \in L^1(X)$, we consider the corresponding localization operator $L_F \colon \H \rightarrow \H$ defined (weakly) as
	\begin{equation}\label{eq : definition localization operator}
		\langle L_F f, g \rangle = \int_X F(x) \langle f,\varphi_x \rangle \langle \varphi_x,g \rangle \, d\mu(x), \quad f,g \in \H.
	\end{equation}
	From \ref{coherent state map is bounded} and since we supposed $F \in L^1(X)$ it is immediate to see that $L_F$ is bounded and that $\| L_F \| \leq c_1^2 \| F \|_{L^1(X)}$. In fact, it turns out that $L_{F}$ is a trace class operator.
	\begin{prop}\label{prop : formula for trace and HS norm with F in L1}
		Under the above assumptions $L_F$ is trace class and its trace is given by
		\begin{equation}\label{eq : formula trace}
			\mathrm{tr} L_F = \int_X F(x) \| \varphi_x \|^2 \, d\mu(x).
		\end{equation}
		Moreover, for its Hilbert-Schmidt norm we have the formula
		\begin{equation}\label{eq : formula norm HS}
			\| L_F \|_{\mathrm{HS}}^2 = \int_{X \times X} F(x) | \langle \varphi_x, \varphi_y \rangle |^2 \overline{F(y)} \, d\mu(x) d\mu(y).
		\end{equation}
	\end{prop}
	\begin{proof}
		We start proving that $L_F$ is a trace class operator and that \eqref{eq : formula trace} holds true. We can suppose, without loss of generality, that $F$ is nonnegative. 
		
		Given an orthonormal basis $f_j$, $j=1,2,\ldots,$ of $\H$, we have
		\begin{align*}
			\sum_{j=1}^{\infty} \langle L_F f_j , f_j \rangle 
			&=\int_X F(x) \sum_{j=1}^{\infty} | \langle f_j, \varphi_x \rangle |^2 \, d\mu(x) \\
			&= \int_X F(x) \| \varphi_x \|^2 \, d\mu(x) \overset{\ref{coherent state map is bounded}}{\leq} c_1^2 \int_X F(x) d \mu(x) = c_1^2 \| F \|_{L^1(X)},
		\end{align*}
		where the exchange between summation and integral is allowed since everything is positive. 
		
		Now we prove \eqref{eq : formula norm HS}. We have
		\begin{align*}
			\| L_F \|_{\mathrm{HS}}^2 &= \sum_{j=1}^{\infty} \langle L_F f_j, L_F f_j \rangle \\
			&= \sum_{j=1}^{\infty} \int_X F(x) \langle f_j, \varphi_x \rangle \langle \varphi_x, L_F f_j \rangle \, d\mu(x) \\
			&= \sum_{j=1}^{\infty} \int_X F(x) \langle f_j, \varphi_x \rangle \overline{ \int_X F(y) \langle f_j, \varphi_y \rangle \langle \varphi_y, \varphi_x \rangle \, d\mu(y) } \, d\mu(x) \\
			&= \sum_{j=1}^{\infty} \int_{X \times X} \langle f_j, \varphi_x \rangle \overline{\langle f_j, \varphi_y \rangle} \langle \varphi_x, \varphi_y \rangle F(x) \overline{F(y)} \, d\mu(x) d\mu(y),
		\end{align*}
		and the desired results follows by exchanging the summation and the integrals and noticing that
		\begin{equation*}
			\sum_{j=1}^{\infty} \langle f_j, \varphi_x \rangle \overline{\langle f_j, \varphi_y \rangle} = \langle \varphi_y, \varphi_x \rangle.
		\end{equation*}
		This exchange is justified since, by the Cauchy-Schwarz inequality,
		\begin{equation*}
			\sum_{j=1}^{\infty} | \langle f_j ,\varphi_x \rangle \overline{\langle f_j, \varphi_y \rangle} | \leq \| \varphi_x \| \| \varphi_y \| \overset{\ref{coherent state map is bounded}}{\leq} c_1^2,
		\end{equation*}
		and therefore
		\begin{align*}
			\sum_{j=1}^{\infty} \int_{X \times X} |\langle f_j ,\varphi_x \rangle| |\langle f_j, \varphi_y \rangle| |\langle \varphi_x, \varphi_y \rangle F(x) F(y)| \, d\mu(x) d\mu(y) \leq c_1^4 \|F\|_{L^1(X)}^2.
		\end{align*}
	\end{proof}

	\begin{prop}\label{prop : norm HS for F in L2}
		Assume, in addition to the hypotheses of Proposition \ref{prop : formula for trace and HS norm with F in L1}, that the following Bessel type inequality holds for every $f \in \H$ and some $c_2 > 0$:
		\begin{equation}\label{eq : Bessel type inequality}
			\int_X | \langle f, \varphi_x \rangle |^2 \, d\mu(x) \leq c_2^2 \| f \|^2.
		\end{equation}
	Then, for $F \in L^2(X)$, $L_F$ is a Hilbert-Schmidt operator and \eqref{eq : formula norm HS} continues to hold.
	\end{prop}
	\begin{proof}
		We notice that the map $F \mapsto L_F$ is bounded from $L^2(X)$ into the space of linear bounded operators on $\mathcal{H}$ -which we will denote by $\mathcal{L}(\H)$. Indeed,
		\begin{equation*}
			\| L_F f \| = \sup_{\|g\| \leq 1} | \langle L_F f, g \rangle | \leq c_1 c_2 \| F \|_{L^2(X)} \| f \|,
		\end{equation*}
		where we used the Cauchy-Schwarz inequality, property \eqref{eq : Bessel type inequality} and the fact that $\|\varphi_x\|\leq c_1$.\par On the other hand, thanks to \eqref{eq : Bessel type inequality} we have
		\begin{equation*}
		\esssup_{y\in X}\int_X |\langle \varphi_x, \varphi_y \rangle|^2 \, d\mu(x) = \esssup_{x\in X}\int_X |\langle \varphi_x, \varphi_y \rangle|^2 \, d\mu(y) \leq c_1^2 c_2^2. 
		\end{equation*}
		Hence, by Schur's test, we see that the right-hand side of \eqref{eq : formula norm HS} is a continuous quadratic form on $L^2(X)$. Therefore, if we take a sequence $F_n \in L^1(X) \cap L^2(X)$ that converges to $F$ in $L^2(X)$, the sequence $L_{F_n}$ is a Cauchy sequence in the space of Hilbert-Schmidt operators on $\H$, and therefore has a limit. Since $L_{F_n} \rightarrow L_F$ in $\L(\H)$ due to the continuity of the map $F \mapsto L_F$, by the uniqueness of the limit we conclude that $L_{F_n}$ converges to $L_{F}$ also in the space of Hilbert-Schmidt operators on $\H$. Hence $L_F$  is a Hilbert-Schmidt operator itself, for which \eqref{eq : formula norm HS} holds.
	\end{proof}
	\begin{remark}
		Assume, in place of \eqref{eq : Bessel type inequality}, the stronger {\it resolution of the identity} formula
		\begin{equation*}\label{eq : resolution of the identity}
			\int_X | \langle f, \varphi_x \rangle |^2 \, d\mu(x) = c \|f\|^2,
		\end{equation*}
		for some $c>0$. Then the linear map $\V : \H \rightarrow L^2(X)$ given by
		\begin{equation*}\label{eq : abstract definition coherent state transform}
			\V f(x) \coloneqq \dfrac{1}{\sqrt{c}} \langle f, \varphi_x \rangle
		\end{equation*}
		is an isometry and its range is a reproducing kernel Hilbert space, i.e.:
		\begin{equation*}\label{eq : formula reproducing kernel Hilbert space}
			\V f(x) = \dfrac{1}{\sqrt{c}} \langle \V f, \V \varphi_x \rangle_{L^2(X)} = \dfrac{1}{c} \int_X \V f(y) \langle \varphi_x, \varphi_y \rangle \, d\mu(y).
		\end{equation*}
		The above formula \eqref{eq : formula norm HS} was proved for particular reproducing kernel Hilbert spaces by several authors (see \cite{abreu} and the references therein), in particular when $F$ is the characteristic function of a subset $\Omega \subset X$ of finite measure. However, Proposition \ref{prop : formula for trace and HS norm with F in L1} shows that no reproducing property is in fact necessary in that case. Also, for Proposition \ref{prop : norm HS for F in L2} we only assumed the Bessel type inequality \eqref{eq : Bessel type inequality}.
	\end{remark}
	
	\subsection{Time-frequency localization operators}
	We are now switching our attention towards more classical time-frequency localization operators. Our measure space $X$ is now $\R^d \times \R^d$ while the Hilbert space $\H$ is now $L^2(\R^d)$. Given the $L^2$-normalized Gaussian $\varphi(x) = 2^{d/4} e^{-\pi |x|^2}$, for any $z = (x_0, \omega_0) \in \R^d \times \R^d$ we consider the following functions
	\begin{equation*}\label{eq : time-frequency shift of normalized Gaussian}
		\varphi_z (x) = e^{2 \pi i \omega_0 \cdot x} \varphi(x-x_0), \quad x \in \R^d.
	\end{equation*}
	With this particular choice, given $f \in L^2(\R^d)$ the map $ \R^{2d}\ni z \mapsto \langle f, \varphi_z \rangle$ is the usual short-time Fourier transform $\mathcal{V}f$ with Gaussian window as defined in \eqref{eq STFT}, which is a continuous and therefore measurable function. Moreover $\| \varphi_z \|_{L^2} = 1$ for every $z \in \R^{2d}$ and $\mathcal{V}:L^2(\R^d)\to L^2(\R^{2d})$ is an isometry (see e.g. \cite{grochenig}). This means that the assumptions of Proposition \ref{prop : norm HS for F in L2} are satisfied and from a direct computation one can see that
	\begin{equation*}\label{eq : kernel for HS norm formula}
		|\langle \varphi_z, \varphi_w \rangle|^2 = e^{-\pi |z-w|^2}, \quad z,w \in \R^{2d}.
	\end{equation*}
	As a consequence of Propositions \ref{prop : formula for trace and HS norm with F in L1} and \ref{prop : norm HS for F in L2} we obtain, for $F \in L^1(\R^{2d}) + L^2(\R^{2d})$,
	\begin{equation}\label{eq HSstima}
	\| L_F \|_{\mathrm{HS}}^2 = \int_{\R^{2d} \times \R^{2d}} F(z) e^{-\pi |z-w|^2} \overline{F(w)} \, dz dw.
	\end{equation}
	We observe that the function $e^{-\pi t^2}$ for $t \geq 0$ is strictly decreasing.
	\begin{prop}\label{prop : HS norm is increasing after rearrangement}
		Let $F \in L^1(\R^{2d}) + L^2(\R^{2d})$. Then
		\begin{equation*}\label{eq : inequality for HS norm after rearrangement}
			\|L_{F}\|_{\mathrm{HS}} \leq \| L_{|F|^*} \|_{\mathrm{HS}},
		\end{equation*}
		where $|F|^*$ is the symmetric decreasing rearrangement of $|F|$. Equality occurs if and only if $F(z) = e^{i \theta} \rho(|z-z_0|)$ for a.e. $z \in \R^{2d}$ for some $\theta \in \R$, $z_0 \in \R^{2d}$ and some decreasing function $\rho \colon [0,+\infty) \rightarrow [0, +\infty)$.
	\end{prop}
	\begin{proof}
		By \eqref{eq HSstima} and Riesz' rearrangement inequality (Theorem \ref{th : Riesz rearrangement inequality}) we have 
		\begin{align*}
			\| L_F \|_{\mathrm{HS}}^2 &= \int_{\R^{2d} \times \R^{2d}} F(z) e^{-\pi |z-w|^2} \overline{F(w)} \, dz dw\\
			&\leq \int_{\R^{2d} \times \R^{2d}} |F(z)| e^{-\pi |z-w|^2} |F(w)| \, dz dw \\
			&\leq \int_{\R^{2d} \times \R^{2d}} |F|^*(z) e^{-\pi |z-w|^2} |F|^*(w) \, dz dw.
		\end{align*}
		The first inequality becomes an equality if and only if
		\begin{equation*}
			F(z) \overline{F(w)} = |F(z)||F(w)|
		\end{equation*}
		for a.e. $z,w \in \R^{2d}$, which means that $F(z) = e^{i \theta} |F(z)|$ a.e. in $\R^{2d}$ for some $\theta \in \R$.\\
		The second inequality is an equality if and only if $|F(z-z_0)|$ is symmetric decreasing for some $z_0 \in \R^{2d}$ (Theorem \ref{th : Riesz rearrangement inequality}).
	\end{proof}
	Given $\Omega \subset \R^{2d}$ of finite measure, we write $L_{\Omega}$ for $L_{\chi_{\Omega}}$. Since $\Omega$ has finite measure, we have $\chi_{\Omega} \in L^1(\R^{2d})$, hence we have the following corollary.
	\begin{cor}\label{cor : inequality for HS norm after rearrangement for characteristic functions}
		Let $\Omega \subset \R^{2d}$ be a subset of finite measure. Then
		\begin{equation*}
			\| L_{\Omega} \|_{\mathrm{HS}} \leq \| L_{\Omega^*} \|_{\mathrm{HS}},
		\end{equation*}
		where $\Omega^* \subset \R^{2d}$ is the open ball centered at 0 and measure $|\Omega^*| = |\Omega|$. Equality occurs if and only if $\Omega$ is (equivalent, up to a set of measure zero, to) a ball.
	\end{cor}
	\begin{remark}
		The quantity $\| L_{B_r} \|_{\mathrm{HS}}$ can be ``explicitly'' computed in terms of Bessel functions:
		\begin{align*}
			\|  L_{B_r} \|_{\mathrm{HS}}^2 &= \int_{\R^{2d}} \chi_{B_r}(z) \left( \int_{\R^{2d}} e^{-\pi |z-w|^2} \chi_{B_r} (w) \, dw \right) \, dz \\
			&= \int_{\R^{2d}} \chi_{B_r}(z) (e^{-\pi |\cdot|^2} * \chi_{B_r}) (z) \, dz \\
			\overset{\mathrm{Parseval}}&{=} \int_{\R^{2d}} |\widehat{\chi_{B_r}}(w)|^2 e^{-\pi |w|^2} \, dw
		\end{align*}
		and $\widehat{\chi_{B_r}}(w)$ is given by (see \cite[page 324]{stein2011functional})
		\begin{equation*}
			\widehat{\chi_{B_r}}(w) = 2 \pi |w|^{-d+1} \int_0^r J_{d-1}(2 \pi |w| R) R^d \, dR,
		\end{equation*}
		where $J_{d-1}$ is the Bessel function of order $d-1$.
	\end{remark}
	\begin{remark}
	\ 
	\begin{itemize}\par
	
	\item[(a)]
		Formula \eqref{eq HSstima} can also be obtained by observing that $L_F$ can be written as pseudodifferential operator with Weyl symbol $a(z)$, $z \in \R^{2d}$, given by
		\begin{equation*}
			a = F * \Phi, \quad \Phi(z) = 2^d e^{-2 \pi |z|^2},\ z \in \R^{2d},
		\end{equation*}
		see e.g. \cite{wong}. Hence,
		\begin{equation*}
			\|L_F\|_{\mathrm{HS}}^2 = \|a\|_2^2  = \int_{\R^{2d}} e^{-\pi |w|^2} |\widehat{F}(w)|^2 \, dw = \int_{\R^{2d}} F(z) e^{-\pi |z-w|^2} \overline{F(w)} \, dzdw.
		\end{equation*}
\item[(b)]
		Formula \eqref{eq HSstima} can be written equivalently as
		\begin{equation*}
			\|L_F\|_{\mathrm{HS}}^2 = \int_{\R^{2d}} (F * e^{-\pi |\cdot|^2})(z) \overline{F(z)} \, dz.
		\end{equation*}
		Therefore, using the Cauchy-Schwarz inequality, Young's inequality and the fact that $\int_{\R^{2d}} e^{-\pi |z|^2} \, dz = 1$ it follows that
		\begin{equation*}
			\|L_F\|_{\mathrm{HS}} \leq \|F\|_2,
		\end{equation*}
		which is a well known result (see e.g. \cite{wong}). However, this strategy also shows that equality can never occur if $F \neq 0$, because equality would imply $\hat{F}(w) = c \hat{F}(w)e^{-\pi |w|^2}$ for some $c \geq 0$. 
		
		Also, observe that
		\begin{equation*}
			\sup_{F \in L^2(\R^{2d}) \setminus \{0\}} \dfrac{\|L_F\|_{\mathrm{HS}}}{\|F\|_2} = 1,
		\end{equation*}
		as one sees by taking $F = \chi_{B_r}$ and letting $r \to +\infty$ (we leave the easy computation to the interested reader).
\item[(c)]
	Consider the so-called Schatten-von Neumann class $\mathcal{S}_p$, $1 \leq p < \infty$, constituted of the compact operators $S$ on $L^2(\R^d)$ whose sequence of singular values $\sigma_j$, $j=1,2,\ldots$, belongs to $\ell^p$, equipped with the norm $\|S\|_{\mathcal{S}_p} \coloneqq \left(\sum_{j=1}^{\infty} \sigma_j^p \right)^{1/p}$. In particular, for $p=2$ we have the class of Hilbert-Schmidt operators, with equal norm. For $p=\infty$ we set $\mathcal{S}_{\infty} = \L(L^2(\R^{d}))$.
		
		It is well known (see e.g. \cite{wong}) that
		\begin{equation}\label{eq : boundedness in the Schatten-von Neumann class}
			\|L_F\|_{\mathcal{S}_p} \leq \|F\|_p
		\end{equation}
		for all $1 \leq p \leq \infty$. In fact, for $1 \leq p \leq \infty$ it holds
		\begin{equation*}
			\sup_{F \in L^p(\R^{2d}) \setminus \{0\}} \dfrac{\|L_F\|_{\mathcal{S}_p}}{\|F\|_p} = 1.
			\end{equation*}
		Indeed, if the above supremum were strictly less than 1 for some $p_0 \in [1,2)$, interpolating with the estimate \eqref{eq : boundedness in the Schatten-von Neumann class} with $p=\infty$ would give
		\begin{equation*}
			\|L_F\|_{\mathcal{S}_2} = \|L_F\|_{\mathrm{HS}} \leq C \|F\|_2,
		\end{equation*}
		for some $C<1$, thus contradicting the previous remark. On the other hand, if the supremum were strictly less than 1 for some $p_0 \in (2,+\infty]$ one argues similarly by interpolating with \eqref{eq : boundedness in the Schatten-von Neumann class} with $p=1$.
\end{itemize}
\end{remark}
	\section{Quantitative estimate}
	In the previous section we proved an estimate for the Hilbert-Schmidt norm of time-frequency localization operators. In particular, for operators of the type $L_{\Omega}$ we proved that $\|L_\Omega\|_{\rm HS}$ is maximized, among all subsets $\Omega$ of a given finite measure, when  $\Omega$ is a ball and the balls are the only maximizers (see Corollary \ref{cor : inequality for HS norm after rearrangement for characteristic functions}). In this section we focus our attention on a quantitative version of Corollary \ref{cor : inequality for HS norm after rearrangement for characteristic functions}. Roughly speaking, we want to prove that the difference $\| L_{\Omega^*} \|_{\mathrm{HS}}^2 - \| L_{\Omega} \|_{\mathrm{HS}}^2$ is bounded from below by some function of the set $\Omega$ which measures how much $\Omega$ differs from a ball, which implies that if the above ``deficit" is small then $\Omega$ is ``almost'' a ball. The notion of $\Omega$ being close to a ball is made precise thanks to the {\it Fraenkel asymmetry index} $\alpha[\Omega]$ as defined in \eqref{eq fraenkel}. 
	
	From \eqref{eq HSstima} we have
	\begin{equation}\label{eq : formula HS norm localization operators with characteristic function}
		\|L_{\Omega} \|_{\mathrm{HS}}^2 = \int_{\Omega \times \Omega} e^{- \pi |z-w|^2} \, dz\,dw,
	\end{equation}
	and in the previous section we used Riesz' rearrangement inequality to prove that the right-hand side increases if $\Omega$ is replaced by $\Omega^*$. To obtain a lower bound for $\| L_{\Omega^*} \|_{\mathrm{HS}}^2 - \| L_{\Omega} \|_{\mathrm{HS}}^2$ we will use a quantitative version of Riesz' rearrangement inequality, that was proved by Christ \cite{christ2017sharpened}; see also Frank and Lieb \cite[Theorem 1]{frank2019note} for a generalization to density functions. 
		\begin{teo}\label{th : Christ theorem}
		Let $\delta \in (0, 1/2)$. Then, there exists a constant $c_{d,\delta}$ such that for all balls $B \subset \R^d$ centered at the origin, all $\Omega \subset \R^d$ such that
		\begin{equation*}\label{eq : condition on the measure of Omega}
			\delta < \dfrac{|B|^{1/d}}{2 |\Omega|^{1/d}} < 1 - \delta
		\end{equation*}
		one has
		\begin{equation}\label{eq : quantitative Riesz inequality}
			\int_{\Omega \times \Omega} \chi_B(x-y) \, dx dy \leq \int_{\Omega^* \times \Omega^*} \chi_B(x-y) \, dx dy - c_{d, \delta} |\Omega|^2 \alpha[\Omega]^2.
		\end{equation}
	\end{teo}
	We have therefore the following result.
	\begin{prop}\label{prop : quantitative estimate for HS norm}
		For every subset $\Omega \subset \R^{2d}$ of positive finite measure it holds
		\begin{equation}\label{eq : quantitative estimate for HS norm}
			\|L_{\Omega}\|_{\mathrm{HS}}^2 \leq \| L_{\Omega^*} \|_{\mathrm{HS}}^2 - c_1 \beta(|\Omega|) \alpha[\Omega]^2,
		\end{equation}
		where
		\begin{equation*}\label{eq : expression of beta}
			\beta(t) = \left\{
			\begin{aligned}
				t^{2+\frac{1}{d}},\quad	     & \text{for } 0 < t \leq 1\\
				t^2 e^{-c_2 t^{1/d}}, \quad  & \text{for } t > 1
			\end{aligned}\right.
		\end{equation*}
		for some constants $c_1, c_2 > 0$ depending only on $d$.
	\end{prop}
	\begin{proof}
		We follow the same strategy as in Frank and Lieb \cite[Theorem 4]{frank2019proof}.
		
		Using the formula \eqref{eq : formula HS norm localization operators with characteristic function} and the fact that $$e^{-\pi t^2} = \int_0^{+\infty} \chi_{(0,R)}(t) 2 \pi R e^{- \pi R^2} \, dR$$ we obtain:
		\begin{align*}
			&\| L_{\Omega^*} \|_{\mathrm{HS}}^2 - \| L_{\Omega} \|_{\mathrm{HS}}^2 \\
			&= \int_{0}^{+\infty} \left( \int_{\Omega^* \times \Omega^*} \chi_{B_R}(z-w) \, dzdw - \int_{\Omega \times \Omega} \chi_{B_R}(z-w) \, dzdw \right) 2\pi R e^{-\pi R^2}\,dR.
		\end{align*}
		Letting
		\begin{equation*}\label{eq : interval for the estimate in the proof}
			I \coloneqq \left\{R>0 \colon \dfrac14 < \dfrac{|B_R|^{1/2d}}{2|\Omega|^{1/2d}} < \dfrac34\right\}
		\end{equation*}
		we can use Christ's result (Theorem \ref{th : Christ theorem}) with $\delta=1/4$, thus obtaining
		\begin{align*}
			\| L_{\Omega^*} \|_{\mathrm{HS}}^2 - \| L_{\Omega} \|_{\mathrm{HS}}^2 &\geq c |\Omega|^2 \alpha[\Omega]^2 \int_I 2 \pi R e^{-\pi R^2} \, dR \\
			&= c |\Omega|^2 \alpha[\Omega]^2 \left( e^{-c'_1 |\Omega|^{1/d}} - e^{-c'_2 |\Omega|^{1/d}} \right),
		\end{align*}
		where $c$ is the constant from Theorem \ref{th : Christ theorem}, and $c'_1 = \frac{\pi}{4|B_1|^{1/d}}$ and $c'_2 =  \frac{9\pi}{4|B_1|^{1/d}}$ are constants that depend only on $d$. 
		
		To highlight the behavior of the latter expression as $|\Omega| \to 0^+$ or $|\Omega| \to + \infty$, in the statement we introduced the function $\beta(t)$ which satisfies $C^{-1}\beta(t)\leq e^{-c'_1 t^{1/d}} - e^{-c'_2 t^{1/d}}\leq C\beta(t)$ for some constant $C>0$ depending in $d$. 
	\end{proof}
	
	\section{Some remarks on the sharpness of Proposition \ref{prop : quantitative estimate for HS norm}}\label{sec : some remarks}
	\subsection{Sharpness of the power $\alpha[\Omega]^2$}
	In this section we prove that the power $\alpha[\Omega]^2$ appearing in \eqref{eq : quantitative estimate for HS norm} is optimal, in the sense that we cannot take any exponent less than 2.
	
	 To this end, for $0 < \varepsilon < 1$ let
	\begin{equation*}
		\Omega_{\varepsilon} \coloneqq \{ z \in \R^{2d} \colon |z| \leq 1-\varepsilon \text{ or } 1 \leq |z| \leq 1+\delta  \},
	\end{equation*}
	where $\delta > 0$ is chosen so that $|\Omega_{\varepsilon}| = |B_1|$. This implies that $\delta = \delta(\varepsilon)$ depends on $\varepsilon$ and from the implicit function theorem we see that
	\begin{equation*}
		\delta(\varepsilon) = \varepsilon + O(\varepsilon^2) \quad \text{as } \varepsilon \rightarrow 0^+.
	\end{equation*}
	Using the formula \eqref{eq : formula HS norm localization operators with characteristic function} and the fact that $\chi_{\Omega_{\varepsilon}} = \chi_{B_1} + \chi_{\Omega_{\varepsilon}} - \chi_{B_1}$ we obtain:
	\begin{align*}
		\| L_{B_1} \|_{\mathrm{HS}}^2 - \| L_{\Omega_{\varepsilon}} \|_{\mathrm{HS}}^2 &= 2 \int_{\R^{2d} \times \R^{2d}} (\chi_{B_1} - \chi_{\Omega_{\varepsilon}})(z) e^{- \pi |z-w|^2} \chi_{B_1}(w) \, dzdw \\
		&- \int_{\R^{2d} \times \R^{2d}} (\chi_{\Omega_{\varepsilon}} - \chi_{B_1})(z) e^{- \pi |z-w|^2} (\chi_{\Omega_{\varepsilon}} - \chi_{B_1})(w) \, dzdw.
	\end{align*}
	The second integral can be easily estimated as follows:
	\begin{equation*}
		\left\lvert \int_{\R^{2d} \times \R^{2d}} (\chi_{\Omega_{\varepsilon}} - \chi_{B_1})(z) e^{- \pi |z-w|^2} (\chi_{\Omega_{\varepsilon}} - \chi_{B_1})(w) \, dzdw \right\rvert \leq |\Omega_{\varepsilon} \triangle B_1|^2 = O(\varepsilon^2),
	\end{equation*}
	while the first integral can be written in the following way:
	\begin{align*}
		\int_{\R^{2d} \times \R^{2d}} (\chi_{B_1} - \chi_{\Omega_{\varepsilon}})(z) e^{- \pi |z-w|^2}  &\chi_{B_1}(w) \, dzdw \\
		&= \int_{\R^{2d}} \left(\int_{B_1} e^{-\pi |z-w|^2} \, dw\right) (\chi_{B_1} - \chi_{\Omega_{\varepsilon}})(z) \, dz.
	\end{align*}
	We notice that the inner function is radial, so letting 
	$$ f(|z|) := \int_{B_1} e^{-\pi |z-w|^2} \, dw$$ and using polar coordinates we see that this integral can be written (up to a multiplicative constant that depends only on $d$) as
	\begin{equation*}
		\int_{1-\varepsilon}^1 f(r)r^{2d-1} \, dr - \int_1^{1+\delta} f(r) r^{2d-1} \, dr.
	\end{equation*}
	Since $f$ is smooth, we have that this difference is equal to
	\begin{equation*}
		\varepsilon f(1) + O(\varepsilon^2) - \delta f(1) + O(\delta^2) \overset{\delta = \varepsilon + O(\varepsilon^2)}{=} O(\varepsilon^2).
	\end{equation*}
	Hence, by \eqref{eq : quantitative estimate for HS norm} we see that, for some $C>0$,
	\begin{equation*}
		C \varepsilon^2\geq \| L_{B_1} \|_{\mathrm{HS}}^2 - \| L_{\Omega_{\varepsilon}} \|_{\mathrm{HS}}^2 \geq c_1 \alpha[\Omega_{\varepsilon}]^2.
	\end{equation*}
	Now, it is easy to see that $\alpha[\Omega_\varepsilon] \geq c_d \varepsilon$ (where $c_d$ is a constant that depends only $d$) and therefore the exponent of $\alpha[\Omega]$ in \eqref{eq : quantitative estimate for HS norm} cannot be replaced by any number smaller than 2.
	
	\subsection{Sharpness of the power $t^{2+1/d}$ for $0<t<1$}
	Having proved that the exponent of $\alpha[\Omega]$ in \eqref{eq : quantitative estimate for HS norm} is optimal, we can prove that the behavior of the function $\beta(t)$ in \eqref{prop : quantitative estimate for HS norm} is optimal for $0 < t < 1$, i.e. the power $t^{2+1/d}$ is sharp.
	
	To this end, we consider any subset $\Omega\subset\R^{2d}$ of positive finite measure, which is {\it not} (equivalent, up to set of measure zero, to) a ball. Then, if we consider the dilation of $\Omega$ by a factor $r>0$, namely $\Omega_r = \{z \in \R^{2d} \colon \frac{z}{r} \in \Omega\}$, we have that  the asymmetry index $\alpha[\Omega_r] $ is not zero and independent of $r$. 
	
	By \eqref{eq : formula HS norm localization operators with characteristic function}, letting $g(|z|) = 1- e^{-\pi |z|^2}$ and using the fact that $|\Omega_r^*| = |\Omega_r|$ we have:
	\begin{align*}
		\| L_{\Omega_r^*}\|_{\mathrm{HS}}^2 - \| L_{\Omega_r} \|_{\mathrm{HS}}^2 &= \int_{\Omega_r^* \times \Omega_r^*} e^{-\pi |z-w|^2} \, dzdw - \int_{\Omega_r \times \Omega_r} e^{-\pi |z-w|^2} \, dzdw \\
		&= \int_{\R^{2d} \times \R^{2d}} g(|z-w|) \left[ \chi_{\Omega_r}(z) \chi_{\Omega_r}(w) -\chi_{\Omega_r^*}(z) \chi_{\Omega_r^*}(w) \right] \, dzdw \\ 
		&\leq \int_{\R^{2d} \times \R^{2d}} g(|z-w|) \chi_{\Omega_r}(z) \chi_{\Omega_r}(w) \, dzdw. \\
		&\leq c r^2 |\Omega_r|^2 = c' |\Omega_r|^{2+1/d},																
	\end{align*}
	where we used the fact that $0 \leq g(r) \leq c r^2$. 
	
	Hence, if \eqref{eq : quantitative estimate for HS norm} holds true then
	\begin{equation*}
		c \beta(|\Omega_r|) \alpha[\Omega_r]^2 \leq \| L_{\Omega_r^*} \|_{\mathrm{HS}}^2 - \| L_{\Omega_r} \|_{\mathrm{HS}}^2 \leq c' |\Omega_r|^{2+1/d},
	\end{equation*}
	which is possible if and only if $\beta(t) \leq c'' t^{2+1/d}$ for $t$ small, since, as already observed, $\alpha[\Omega_r]>0$ is independent of $r$.
	
	\subsection{The behavior of $\beta(t)$ as $t \rightarrow +\infty$}
	In this section we are concerned with the behavior of the function $\beta(t)$ in Proposition \ref{prop : quantitative estimate for HS norm} as $t \rightarrow +\infty$, which is probably non-sharp. 
	
	Precisely, we claim that if the the inequality
	\begin{equation}\label{eq : quantitative estimate for HS norm bis}
		\|L_{\Omega}\|_{\mathrm{HS}}^2 \leq \| L_{\Omega^*} \|_{\mathrm{HS}}^2 - c_1 \beta(|\Omega|) \alpha[\Omega]^2
	\end{equation}
	holds for some constant $c>0$ and some function $\beta(t)$, and every subset $\Omega\subset\mathbb{R}^{2d}$ of finite large measure, then 
	$$\beta(t) \leq ct^{2-1/2d}\qquad {\rm as}\ t \rightarrow +\infty.$$
	
	Indeed, fixed any $v \in \R^{2d}$ with $|v|=1$ and let $\Omega_r = \Omega_r^1 \cup \Omega_r^2$, where
	\begin{align*}
		\Omega_r^1 &= \left\{z \in \R^{2d} \colon \frac{r}{3} \leq |z| \leq r \right\},\\
		\Omega_r^2 &= \left\{z \in \R^{2d} \colon |z - 2rv| < \frac{r}{3} \right\}
	\end{align*}
	(see Figure \ref{im : set Omega}).
	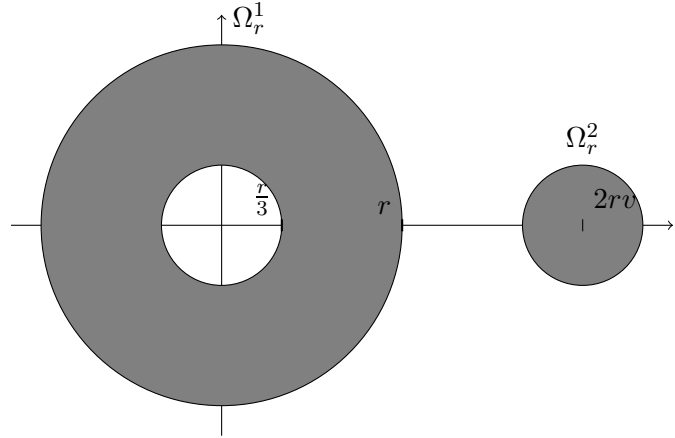
\begin{figure}[H]
		\centering
		\begin{tikzpicture}[scale=0.4]
			\draw[->] (-7, 0) -- (15, 0);
			\draw[->] (0, -7) -- (0, 7);
			\draw [fill=gray, opacity=0.50, even odd rule] (0,0) circle (2) circle (6);
			\draw [fill=gray, opacity=0.50] (12,0) + (0,0) circle(2);
			\draw (12,-0.2) -- (12, 0.2) node[above right] {\small $2rv$};
			\draw (12,2) node[above] {\small $\Omega_r^2$};
			\draw (0, 6) node[above right] {\small $\Omega_r^1$};
			\draw [thick] (2,-0.2) -- (2, 0.2);
			\draw (2, 0) node[above left] {\small $\frac{r}{3}$};
			\draw [thick] (6, -0.2) -- (6, 0.2);
			\draw (6, 0) node[above left] {\small $r$};
			
		\end{tikzpicture}
		\caption{Representation of $\Omega_r$ with $d=1$ and $v = e_1$.}
		\label{im : set Omega}
	\end{figure}
	For the sake of brevity, we let
	\begin{equation*}
		I(f,g) := \int_{\R^{2d} \times \R^{2d}} f(z) e^{-\pi |z-w|^2} \overline{g(w)} \, dzdw.
	\end{equation*}
	Using \eqref{eq : formula HS norm localization operators with characteristic function} we have:
	\begin{equation*}
		\| L_{\Omega_r^*} \|_{\mathrm{HS}}^2 - \| L_{\Omega_r} \|_{\mathrm{HS}}^2 = I(\chi_{\Omega_r^*}, \chi_{\Omega_r^*}) - I(\chi_{\Omega_r}, \chi_{\Omega_r}).
	\end{equation*}
	We notice that $\Omega_r^* = B_r$ so, using the fact that $\chi_{B_r} = \chi_{B_{r/3}} + \chi_{\Omega_r^1}$, $\chi_{\Omega_r} = \chi_{\Omega_r^1} + \chi_{\Omega_r^2}$ and that $I(\chi_{\Omega_r^2}, \chi_{\Omega_r^2}) = I(\chi_{B_{r/3}}, \chi_{B_{r/3}})$ we obtain
	\begin{equation*}
		\| L_{\Omega_r^*} \|_{\mathrm{HS}}^2 - \| L_{\Omega_r} \|_{\mathrm{HS}}^2 = 2 I(\chi_{\Omega_r^1}, \chi_{B_{r/3}}) - 2I(\chi_{\Omega_r^1}, \chi_{\Omega_r^2}) \leq 2 I(\chi_{\Omega_r^1}, \chi_{B_{r/3}}).
	\end{equation*}
	%We point out that the estimate is quite good since $I(\chi_{\Omega_r^1}, \chi_{\Omega_r^2}) \leq 2 I(\chi_{\Omega_r^1}, \chi_{B_{r/3}}) \leq c_1 e^{-c_2|\Omega_r|^{1/d}} |\Omega_r|^2$. 
	The function in the right-hand side can be written in the following way:
	\begin{equation*}
		I(\chi_{\Omega_r^1}, \chi_{B_{r/3}}) = \int_{\R^{2d}} e^{-\pi |z|^2} \left( \chi_{\Omega_r^1} * \chi_{B_{r/3}}\right)(z) \, dz.
	\end{equation*}
	The function $\chi_{\Omega_r^1} * \chi_{B_{r/3}}$ is radial and it holds
	\begin{equation*}
		(\chi_{\Omega_r^1} * \chi_{B_{r/3}})(z) \leq c r^{2d-1} \min\left\{|z|, \frac{2r}{3}\right\},
	\end{equation*}
	see Figure \ref{im : bound convolution} for a graphical intuition.
	\begin{figure}[H]
		\centering
		\begin{tikzpicture}[scale=0.4]
			\draw[->] (-7, 0) -- (7, 0);
			\draw[->] (0, -7) -- (0, 7);
			\draw [fill=gray, opacity=0.50, even odd rule, dashed] (0,0) circle (2) (1,0) circle (2);
			\draw [fill=white, dashed] (0,0) circle (2);
			\draw (0,2) node[above left] {\small $B_{r/3}$};
			\draw (1,0) circle (2);
			\draw (2,2) node[above] {\small $z + B_{r/3}$};
			\draw [fill=black] (0,0) circle (1pt);
			\draw [fill=black] (1,0) circle (1pt);
			\draw [fill=gray, opacity=0.25, dashed, even odd rule] (0,0) circle (2) (0,0) circle (6);
			\draw (0,0) circle (6);
			\draw (0,6) node[above right] {$\Omega_r^1$};
			\draw[->] (0,0) -- (1,0);
			\draw (0.5, 0) node[above] {$z$};		
		\end{tikzpicture}
		\caption{The area of the darker gray region is $(\chi_{\Omega_r^1} * \chi_{B_{r/3}})(z)$}
		\label{im : bound convolution}
	\end{figure}
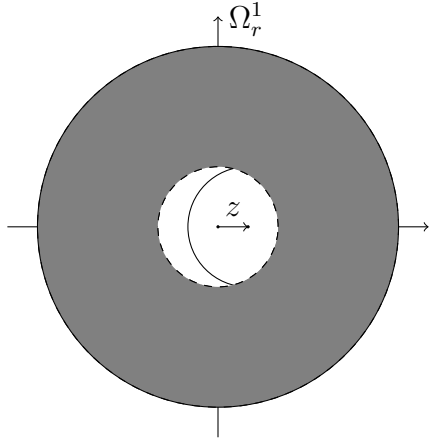
	This implies that
	\begin{equation*}
		I(\chi_{\Omega_r^1}, \chi_{B_{r/3}}) \leq cr^{2d-1} \int_{\R^{2d}} |z| e^{-\pi |z|^2} \, dz \leq c' r^{2d-1} \leq c'' |\Omega_r|^{1-1/2d}.
	\end{equation*}
	Since the Fraenkel index $\alpha[\Omega_r]$ is not zero and independent of $r$ the above claim is proved. 
	
	One can see that the above example also ``saturates" the analogous quantitative bound in $\R^d$ for the interaction kernel $|x|^{-\lambda}$, with $0<\lambda<d$ \cite[Theorem 4]{frank2019proof}. These facts and some further experimentation suggest the following conjecture.
	
	\begin{conj}\label{conj 1}
		For every subset $\Omega \subset \R^{2d}$ of finite positive measure it holds
		\begin{equation}\label{eq : quantitative estimate for HS norm conj}
			\|L_{\Omega}\|_{\mathrm{HS}}^2 \leq \| L_{\Omega^*} \|_{\mathrm{HS}}^2 - c \tilde{\beta}(|\Omega|) \alpha[\Omega]^2,
		\end{equation}
		where
		\begin{equation*}\label{eq : expression of beta conj}
			\tilde{\beta}(t) = \left\{
			\begin{aligned}
				t^{2+\frac{1}{d}},\quad & \text{for } 0 < t \le 1\\
				t^{2-1/2d}, \quad & \text{for } t > 1
			\end{aligned}\right.
		\end{equation*}
		for some constant $c> 0$ depending only on $d$.
	\end{conj}
	In view of the above connection with Christ's result (Theorem \ref{th : Christ theorem}), this also suggests the following conjecture, of independent interest. 	We observe that Frank and Lieb arrived at the same conjecture when working on their papers \cite{frank2019note,frank2019proof} (private communication). 
	
	\begin{conj}\label {conj 2}
		Let $0<\delta<1$. There exists a constant $c_{d,\delta}>0$ such that, for all balls $B\subset\mathbb{R}^d$ centered at the origin, and all subsets $\Omega\subset \mathbb{R}^d$ of finite measure satisfying 
		\[
		\frac{|B|^{1/d}}{2|\Omega|^{1/d}}\leq 1- \delta
		\]
		we have 
		\[
		\frac{1}{2}\int_{\Omega\times \Omega} \chi_B(x-y)\, dx\, dy\leq \frac{1}{2}\int_{\Omega^\ast\times \Omega^\ast} \chi_B(x-y)\, dx\, dy-c_{d,\delta} (|B|/|\Omega|)^{1+1/d} |\Omega|^2\,\alpha[\Omega]^2.
		\]
		\par\medskip 
	\end{conj}
	Indeed, arguing as in the proof of Proposition \ref{prop : quantitative estimate for HS norm} it is easy to see that Conjecture \ref{conj 2} implies Conjecture \ref{conj 1}. Also, in dimension $d=1$ Conjecture \ref{conj 2} was positively solved by Christ \cite{christ2019one}.
	
\begin{remark}\label{rem frank} Rupert Frank pointed out to us that  Conjecture \ref{conj 2} easily implies the sharp  isoperimetric inequality in quantitative form (first proved by Fusco, Maggi and Pratelli \cite{fusco2008} by rearrangement techniques), that is
\[
P(\Omega)\geq P(\Omega^\ast)+c_d |\Omega|^{\frac{d-1}{d}}\alpha[\Omega]^2
\]
for $0<|\Omega|<\infty$, where $P(\Omega)$ is the (distributional) perimeter of $\Omega$.

Indeed, it follows from \cite{davila2002} that for every bounded measurable set $\Omega\subset\R^d$ of finite perimeter we have 
	\begin{equation}\label{unomenos}
	\lim_{s\to 1^-} (1-s)P_s(\Omega)=K_d P(\Omega)
	\end{equation}
	where, for $s\in (0,1)$, 
	\[
	P_s(\Omega):=\int_\Omega\int_{\Omega^c} \frac{1}{|x-y|^{d+s}}\, dx\, dy
	\]
	is the fractional $s$-perimeter of $\Omega$ and $K_d>0$ is a constant depending only on the dimension. Hence it is sufficient to obtain a suitable quantitative isoperimetric inequality for $P_s$, that is uniform in $s$ as $s\to 1^-$. To this end, one can argue as in the proof of Proposition \ref{prop : quantitative estimate for HS norm}. Precisely, using the conjecture one arrives at integrating $R^{-s}$ on the interval $(0,c'_d|\Omega|^{1/d})$, which gives 
	\[
	P_s(\Omega)\geq P_s(\Omega^\ast)+\frac{c_d}{1-s} |\Omega|^{\frac{d-s}{d}}\alpha[\Omega]^2
	\]
	uniformly with respect to $s\in (0,1)$, namely with an (explicit) constant $c_d>0$ independent of $s$, which is precisely what we need. 
	
	We observe that the latter inequality was proved by Figalli, Fusco, Maggi, Millot and Morini \cite{figalli2015isoperimetry} for $s\in (s_0,1)$, for any $s_0\in (0,1)$, with a (non-explicit) constant $c(d,s_0)$ depending also on $s_0$ (in place of $c_d$ above). 
\end{remark}
	
	%\section{Analogous results for Schatten $p$-norm}
	
	\section{The Hilbert-Schmidt norm of wavelet localization operators}\label{sec : HS norm of wavelet localization operators}
	In this section we are going to prove that a result analogous to Corollary \ref{cor : inequality for HS norm after rearrangement for characteristic functions} holds also for wavelet localization operators. Just for this section, we consider the following normalization for the Fourier transform (which is a common choice in the wavelet literature): 
	\begin{equation*}
		\widehat{f}(\omega) = \dfrac{1}{\sqrt{2 \pi}} \int_{\R} f(t) e^{-i \omega t} \, dt.
	\end{equation*}
	The role that was previously played by the Gaussian window is now taken by the so-called Cauchy wavelet $\psi_{\beta}$, which for $\beta>0$ is defined by
	\begin{equation}\label{eq : definition Cauchy wavelet}
		\widehat{\psi_{\beta}} (\omega) = \dfrac{1}{c_\beta} \chi_{[0,+\infty)}(\omega) \omega^{\beta} e^{-\omega},
	\end{equation}
	where $c_{\beta} > 0$ is given by $c_{\beta}^2 = 2\pi 2^{-2\beta} \Gamma(2\beta)$ and is chosen so that $\| \widehat{\psi_{\beta}} \|^2_{L^2(\R_+, d\omega/\omega)} = 1/(2\pi)$ (where $\R_+ = (0,+\infty))$. In this context, the Hilbert space $\H$ is given by the Hardy space $H^2(\R)$, that is the space of functions of $L^2(\R)$ whose Fourier transform is supported in $[0,+\infty)$, endowed with the $L^2$-norm. In particular, it holds that $\psi_{\beta} \in H^2(\R)$ for every $\beta > 0$. The ``coherent states'' here are given by
	\begin{equation*}
		\varphi_z(t) := \pi(z) \psi_{\beta}(t) \coloneqq \dfrac{1}{\sqrt{s}} \psi_{\beta} \left(\dfrac{t-x}{s}\right), \quad t \in \R,
	\end{equation*}
	with $z = (x,s) \in \R \times \R_+$, regarded as the hyperbolic space introduced in Section \ref{sec : notation and preliminaries}. We remark that $\pi(z)$ is a unitary representation on $H^2(\R)$ of the affine (``$ax+b$'') group. 
	The transform related with these coherent states is the \emph{wavelet transform} $\W_{\psi_{\beta}}$, which for $f \in H^2(\R)$ is given by
	\begin{equation*}
		\W_{\psi_{\beta}} f (z) \coloneqq \dfrac{1}{\sqrt{s}} \int_{\R} f(t) \overline{\psi_{\beta}\left(\dfrac{t-x}{s}\right)} \, dt, \quad z = (x,s) \in \R \times \R_+.
	\end{equation*} 
	With the above normalization of $\psi_{\beta}$, $\W_{\psi_{\beta}} : H^2(\R) \rightarrow L^2(\R \times \R_+, d\nu)$ is an isometry, with $d\nu=dx\, ds/s^2$.
	
	Now, in this context, the kernel appearing in \eqref{eq : formula norm HS} is given by
	\begin{equation*}
		|\langle \varphi_z, \varphi_w \rangle|^2 = |\langle \psi_{\beta}, \pi(z^{-1} \circ w) \psi_{\beta} \rangle|^2, \quad z,w \in \R \times \R_+,
	\end{equation*}
	where $\circ$ is the product in the affine group ($z \circ w = (x,s) \circ (y,t) = (x + ys, st)$) while $z^{-1}$ is the inverse of $z$ in the group ($z^{-1} = (-x/s, 1/s)$).

	An essential property for the proof of Proposition \ref{prop : HS norm is increasing after rearrangement} was that the function $e^{-\pi t^2}$ appearing in the kernel of \eqref{eq HSstima} is strictly decreasing. A similar fact holds true in this context, namely $
	|\langle \psi_{\beta}, \pi(z) \psi_{\beta} \rangle|^2 = |\W_{\psi_{\beta}} \psi_{\beta}(z)|^2
	$ is strictly symmetric decreasing  around $(0,1)$ (unit of the group) with respect to the Poincaré metric of $\R \times \R_+$.
	
	Indeed, we have:
	\begin{align*}
		\W_{\psi_{\beta}} \psi_{\beta}(z) &= \dfrac{1}{\sqrt{s}} \int_{\R} \psi_{\beta}(t) \overline{\psi_{\beta}\left(\dfrac{t-x}{s}\right)} \, dt \\
		&= \sqrt{s}\int_{\R} \widehat{\psi_{\beta}}(\omega) \overline{e^{-i x \omega}  \widehat{\psi_{\beta}}(s\omega)} \, d\omega \\
		&= \dfrac{s^{\beta + \frac12}}{c_{\beta}^2} \int_0^{+\infty} \omega^{2\beta} e^{-[(1+s) - ix]\omega} \, d\omega \\
		&= \dfrac{s^{\beta + \frac12}}{c_{\beta}^2} [(1+s) - ix]^{-2\beta-1} \Gamma(2\beta + 1),
	\end{align*}
	hence, for some $C > 0$, it holds
	\begin{equation*}
		|\W_{\psi_{\beta}} \psi_{\beta}(z)|^2 = C s^{2\beta + 1} \left[(1+s)^2 + x^2\right]^{-2\beta - 1}.
	\end{equation*}
	On the other hand, identifying $\R\times\R_+$ with $\C_+$ via $z=x+is$, we have
	\begin{equation*}
		\dfrac{4s}{(1+s)^2 + x^2} = 1 - \left\lvert \dfrac{z-i}{z+i} \right\rvert^2,
	\end{equation*}
	therefore 
	\begin{equation*}
		|\W_{\psi_{\beta}} \psi_{\beta}(z)|^2 = \varrho(d_H(z,i)),
	\end{equation*}
	where $\varrho(t) = C[1-\tanh(t/2)]$ is a strictly decreasing function $[0,+\infty)\to [0,+\infty)$. 
	
	Hence, for the kernel in \eqref{eq : formula norm HS} we have
	\begin{equation*}
		|\langle \varphi_z, \varphi_w \rangle|^2 = |\langle \psi_{\beta}, \pi(z^{-1} \circ w) \psi_{\beta} \rangle|^2 = |\W_{\psi_{\beta}}(z^{-1} \circ w)|^2 = \varrho(d_H(z,w)).
	\end{equation*}
	We can therefore state the analog of Proposition \ref{prop : HS norm is increasing after rearrangement} for wavelet localization operators. The proof is exactly the same, with the only difference that one has to use Theorem \ref{th : Riesz rearrangement inequality hyperbolic} in place of Theorem \ref{th : Riesz rearrangement inequality}.  
	\begin{prop}\label{prop : HS norm is increasing after rearrangement wavelet case}
		Let $F \in L^1(\C_+, \nu) + L^2(\C_+, \nu)$. Then
		\begin{equation*}\label{eq : inequality for HS norm after rearrangement wavelet case}
			\|L_{F}\|_{\mathrm{HS}} \leq \| L_{|F|^*} \|_{\mathrm{HS}},
		\end{equation*}
		where $|F|^*$ is the symmetric-decreasing rearrangement of $|F|$. Equality occurs if and only if $F(z) = e^{i \theta} |F|^*(az+b)$ for almost every $z \in \C_+$ and some $a>0$ and $b \in\R$.
	\end{prop}
	\section*{Acknowledgments}
	We would like to thank Rupert Frank for interesting discussions and for his observations about Conjecture \ref{conj 2} in Remark \ref{rem frank}.

	\appendix
	\section{Remarks on the uniqueness of the extremizers in Riesz' rearrangement inequality}
	In Theorem \ref{th : Riesz rearrangement inequality} the cases of equality in Riesz' rearrangement inequality were characterized under the hypothesis that $g$ is \emph{strictly} decreasing (see also \cite{burchard1996cases} for a complete study of the cases of equality). We also stressed this fact in the proof of Proposition \ref{prop : HS norm is increasing after rearrangement}. 
	
	In this appendix, for the benefit of the non-expert reader,  we would like to clarify the necessity of this hypothesis by showing that if $g$ is not strictly decreasing we can {\it always} find characteristic functions $f$ and $h$ that achieve equality in \eqref{eq : Riesz rearrangement inequality} even though $f$ (say) is not, up to a translation, symmetric decreasing.
	
	If we assume that $g$ is decreasing but not strictly then there exist $0 \leq r < R$ such that $g$ is constant on the annulus $B_R \setminus B_r$ (where $B_r = \emptyset$ if $r=0$). Then, we consider the following set
	\begin{equation*}
		\Omega = \{x \in \R^d \colon |x| < r + 2\delta,\, R-4\delta < |x| < R - 2\delta\},
	\end{equation*}
	where $\delta$ is small enough to have $r+2\delta < R - 4\delta$ and we set \[
	f = \chi_{\Omega},\qquad h = \chi_{B_\delta}.
	\]
	
	\begin{figure}[h]
		\centering
		\begin{tikzpicture}[scale=1]
			\coordinate (T) at (-7,0);
			\coordinate (y) at (0.07, 0.04);
			\draw[->, shift={(T)}] (-3.7,0) -- (3.7,0);
			\draw[->, shift={(T)}] (0,-0.5) -- (0,2.5);
			\draw[blue, shift={(T)}] plot[domain=-1:1, samples=100] (\x, {2-(\x)^2});
			\draw[blue, shift={(T)}] (-2,1) -- (-1,1);
			\draw[blue, shift={(T)}] (1,1) -- (2,1);
			\draw[blue, shift={(T)}] plot[domain=2:3.5, samples=100] (\x, {exp(-\x+2)});
			\draw[blue, shift={(T)}] plot[domain=-3.5:-2, samples=100] (\x, {exp(\x+2)});
			\draw[shift={(T)}] (1, 0.1) -- (1, -0.1) node[below] {$r$};
			\draw[shift={(T)}] (2, 0.1) -- (2, -0.1)node[below] {$R$};
			\draw[shift={(T)}] (0,2) node[above right] {$g$};
			%%%%%%%%%%%%%%%%%%%%%%%%%%%%%%%%%%%%%%%%%%%%%%%%%%%%%%%%%%%%%%%%%%%%%
			%\draw[opacity=0.5, fill=gray, even odd rule] (0,0) circle (1) (0,0) circle (2);
			%delta=0.1;
			%\draw[fill opacity=0.25, fill=blue, draw=blue, even odd rule, shift={(y)}] (0,0) circle (1.2) (0,0) circle (1.6) (0,0) circle (1.8);
			\draw[fill opacity=0.5, fill=gray, draw=black, even odd rule] (0,0) circle (1.2) (0,0) circle (1.6) (0,0) circle (1.8);
			\draw[dashed] (0,0) circle (1);
			\draw[dashed] (0,0) circle (2);
			\draw[|<->|] (-1,0) -- (0,0);
			\draw (-0.5,0) node[above] {$r$};
			\draw[|<->|] (0,0) -- (2,0);
			\draw (1,0) node[above left] {$R$};
			\draw[->] (-2.5,0) -- (2.5,0);
			\draw[->] (0,-2.5) -- (0,2.5);		
		\end{tikzpicture}
		\caption{Example of $g$ that is symmetric decreasing but not strictly (left, where $g(x_1,0)$ is represented) and the corresponding set $\Omega$ in dimension $d=2$ (right).}
		\label{im : set Omega appendix}
	\end{figure}
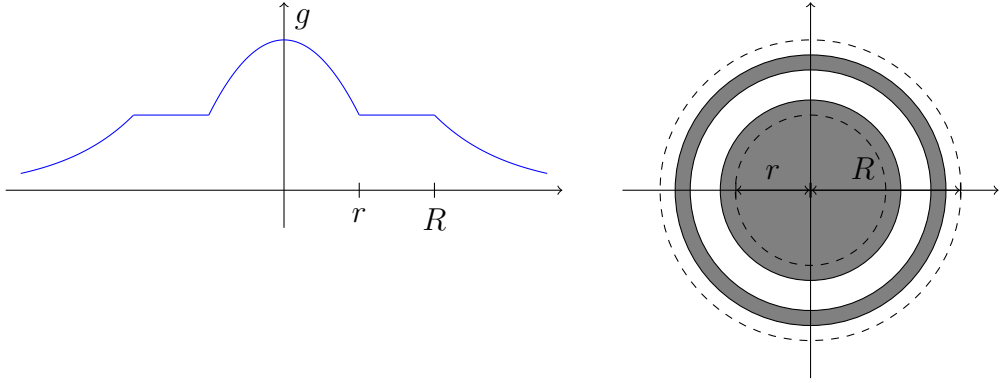
	Since $h(-y)=h(y)$, the left-hand side of \eqref{eq : Riesz rearrangement inequality} can be written as follows:
	\begin{equation*}
		\int_{\R^d} \int_{\R^d} f(x) g(x-y) h(y) \, dx dy = \int_{\R^d} g(x) (f\ast h)(x)\, dx. 	\end{equation*}
	Now we observe that (obviously) $h^\ast=h$ and that the functions $(f\ast h)(x)$ and $(f^\ast \ast h^\ast)(x)$\medskip
	\begin{itemize}
	\item both vanish on $\R^d\setminus B_R$;\medskip
	\item coincide on $B_r$; \medskip 
	\item have the same integral: $\int_{\R^d} (f\ast h)(x)\, dx= \int_{\R^d} (f^\ast\ast h^\ast)(x)\, dx=|B_\delta|\cdot|\Omega|$.
	\end{itemize}\medskip
	As a consequence, we have 
	\[
	\int_{B_R\setminus B_r} (f\ast h)(x)\, dx= \int_{B_R\setminus B_r} (f^\ast\ast h^\ast)(x)\, dx.
	\]
Hence, since $g(x)=c$ (constant) for $x\in B_R\setminus B_r$, we have

	\begin{align*}
		\int_{\R^d} \int_{\R^d} f(x) g(x-y) h(y) \, dx dy &=\int_{B_r} g(x) (f\ast h)(x)\, dx+ \int_{B_R\setminus B_r} g(x) (f\ast h)(x)\, dx\\
		&= \int_{B_r} g(x) (f\ast h)(x)\, dx+ c \int_{B_R\setminus B_r} (f\ast h)(x)\, dx\\
		&= \int_{B_r} g(x) (f\ast h)(x)\, dx+ c \int_{B_R\setminus B_r} (f^\ast\ast h^\ast)(x)\, dx\\
		&= \int_{B_r} g(x) (f\ast h)(x)\, dx+ \int_{B_R\setminus B_r}g(x) (f^\ast\ast h^\ast)(x)\, dx\\
		&= \int_{\R^{d}} \int_{\R^d} f^*(x) g(x-y) h^*(y) \, dx\,dy.
	\end{align*}
	%We have thus found two characteristic functions $f$ and $h$ that achieve equality in \eqref{eq : Riesz rearrangement inequality} even though $f$ is not (up to translations) symmetric decreasing.
	
%	\bibliographystyle{abbrv}
%	\nocite{*}
%	\bibliography{bibliografia}

\end{document}